\documentclass{amsart}

\input xy
\xyoption{all}

\newcommand{\A}{\mathcal{A}}
\newcommand{\F}{\mathcal{F}}
\newcommand{\T}{\mathcal{T}}

\newcommand{\R}{\mathbb{R}}
\newcommand{\C}{\mathbb{C}}
\newcommand{\B}{\mathbb{B}}

\newtheorem{mainthm}{Main Theorem}

\newtheorem{theorem}{Theorem}[section]
\newtheorem{lemma}[theorem]{Lemma}
\newtheorem{proposition}[theorem]{Proposition}

\theoremstyle{definition}

\theoremstyle{remark}

\numberwithin{equation}{section}

\begin{document}

\title[Toeplitz operators and Lagrangian foliations]{Commutative
  algebras of Toeplitz operators and Lagrangian foliations}
\author{R. Quiroga-Barranco} 
\address{Centro de Investigaci\'on en Matem\'aticas,
Apartado Postal 402, 36000, Guanajuato, Guanajuato, M\'exico.}
\email{quiroga@cimat.mx} 
\dedicatory{To Nikolai Vasilevski on the occasion of his 60th birthday}
\thanks{Research supported by CONACYT, CONCYTEG and SNI}
\keywords{Toeplitz operators, Lagrangian submanifolds, foliations}
\subjclass[2000]{47B35, 32M10, 57R30, 53D05}

\maketitle

\begin{abstract}
  Let $D$ be a homogeneous bounded domain of $\C^n$ and $\A$ a set of
  (anti--Wick) symbols that defines a commutative algebra of Toeplitz
  operators on every weighted Bergman space of $D$. We prove that if
  $\A$ is rich enough, then it has an underlying geometric structure
  given by a Lagrangian foliation.
\end{abstract}

\section{Introduction}

In recent work, Vasilevski and his collaborators discovered unexpected
commutative algebras of Toeplitz operators acting on weighted Bergman
spaces on the unit disk (see \cite{Vas2001} and \cite{Vas2003}). It
was even possibly to classify all such algebras as long as they are
assumed commutative for every weight and suitable richness conditions
hold; the latter ensure that the (anti--Wick) symbols that define the
Toeplitz operators have enough infinitesimal linear independence. We
refer to \cite{GQV05} for further details.

Latter on, it was found out that the phenomenon of the existence of
nontrivial commutative algebras of Toeplitz operators extends to the
unit ball $\B^n$ in $\C^n$. There exists at least $n+2$ nonequivalent
such commutative algebras which were exhibited explicitly in
\cite{QV07} and \cite{QV08}. The discovery of such commutative
algebras of Toeplitz operators on $\B^n$ was closely related to a
profound understanding of the geometry of this domain. The description
of these commutative algebras of Toeplitz operators involved
Lagrangian foliations, i.e.~by Lagrangian submanifolds (see
Section~\ref{sec:geom} for detailed definitions), with distinguished
geometric properties.

In order to completely understand the commutative algebras of Toeplitz
operators on $\B^n$ and other domains, it is necessary to determine
the general features that produce such algebras. In particular, there
is the question as to whether or not the commutative algebras of
Toeplitz operators have always a geometric origin. More precisely,
whether or not they are always given by a Lagrangian foliation. 

The main goal of this paper is to prove that, on a homogeneous domain,
commutative algebras of Toeplitz operators are always obtained from
Lagrangian foliations. This is so at least when the commutativity
holds on every weighted Bergman space and a suitable richness
condition on the symbols holds. To state this claim we use the
following notation. We denote by $C^\infty(M)$ the space of smooth
complex-valued functions on a manifold $M$ and by $C^\infty_b(M)$ the
subspace of bounded functions. For a given foliation $\F$ of a
manifold $M$, we will denote by $\A_\F(M)$ the vector subspace of
$C^\infty(M)$ that consists of those functions which are constant
along every leaf of $\F$. Our main result is the following.

\begin{mainthm}
  Let $D$ be bounded domain in $\C^n$ and $\A$ a vector subspace of
  $C^\infty_b(D)$ which is closed under complex conjugation. If for
  every $h \in (0,1)$ the Toeplitz operator algebra $\T_h(\A)$, acting
  on the weighted Bergman space $\A_h^2(D)$, is commutative and $\A$
  satisfies the following richness condition:
  \begin{itemize}
  \item for some closed nowhere dense subset $S \subseteq D$ and for
    every $p \in D \setminus S$ there exist real-valued elements $a_1,
    \dots, a_n \in \A$ such that $d{a_1}_p, \dots, d{a_n}_p$ are
    linearly independent over $\R$,
  \end{itemize}
  then, there is a Lagrangian foliation $\F$ of $D \setminus S$ such
  that $\A|_{D\setminus S} \subseteq \A_\F(D\setminus S)$. In other
  words, every element of $\A$ is constant along the leaves of $\F$.
\end{mainthm}

To obtain this result we prove the following characterization of
spaces of functions which define commutative algebras with respect to
the Poisson brackets in a symplectic manifold.

\begin{theorem}\label{thm:lag-fol}
  Let $M$ be a $2n$-dimensional symplectic manifold and $\A$ a vector
  subspace of $C^\infty(M)$ which is closed under complex conjugation.
  Suppose that the following conditions are satisfied:
  \begin{enumerate}
  \item $\A$ is a commutative algebra for the Poisson brackets of $M$,
    i.e.~$\{a,b\} = 0$ for every $a,b \in \A$, and
  \item for every $p \in M$ there exist real-valued elements $a_1,
    \dots, a_n \in \A$ 
    such that $d{a_1}_p, \dots, d{a_n}_p$ are linearly independent
    over $\R$.
  \end{enumerate}
  Then, there is a Lagrangian foliation $\F$ of $M$ such that $\A
  \subseteq \A_\F(M)$.
\end{theorem}

This together with Berezin's correspondence principle (see
Section~\ref{sec:Ber}) allows us to prove the Main Theorem.

We note that it is a known fact that for a Lagrangian foliation $\F$
in a symplectic manifold $M$, the vector subspace $\A_\F(M)$ is a
commutative algebra with respect to the Poisson brackets. More
precisely we have the following result.

\begin{proposition}\label{prop:abelian}
  If $\F$ is a Lagrangian foliation of a symplectic manifold $M$,
  then:
  $$
  \{a,b\} = 0,
  $$
  for every $a,b \in \A_\F(M)$.
\end{proposition}

Hence, Theorem~\ref{thm:lag-fol} can also be thought of as a converse of
Proposition~\ref{prop:abelian}.

\section{Preliminaries on symplectic geometry and foliations}
\label{sec:geom}
The goal of this section is to establish some notation and state some
very well known results on symplectic geometry and foliations.

For the next remarks on symplectic geometry we refer to
\cite{McDuff-Salamon} for further details.

Let $M$ be a symplectic manifold with symplectic form $\omega$. Being
nondegenerate, the symplectic form $\omega$ defines an isomorphism
between the tangent and cotangent spaces. More precisely, we have the
following elementary fact.

\begin{lemma}\label{lem:iso-omega}
  For every $p \in M$ the map:
  \begin{eqnarray*}
    T_p M &\rightarrow& T_p^* M \\
    v &\rightarrow& \omega(v,\cdot),
  \end{eqnarray*}
  is an isomorphism of vector spaces.
\end{lemma}

This remark allows us to construct vector fields associated to
$1$-forms. In particular, for a complex-valued smooth function $f$
defined over $M$, we define the Hamiltonian field associated to $f$ as
the smooth vector field over $M$ that satisfies the identity:
\begin{equation*}
  df(X) = \omega(X_f,X),
\end{equation*}
for every vector field $X$ over $M$. The Poisson brackets of two
complex-valued smooth functions $f,g$ over $M$ is then given as the
smooth function:
\begin{equation*}
  \{f,g\} =\omega(X_f,X_g) = df(X_g).
\end{equation*}
The following well known result relates the Poisson brackets on smooth
functions to the Lie brackets of vector fields.

\begin{lemma}\label{lem:Lie}
  The Poisson brackets define a Lie algebra structure on the space
  $C^\infty(M)$. Also, we have the identity:
  \begin{equation*}
    [X_f,X_g] = X_{\{f,g\}},
  \end{equation*}
  for every $f,g \in C^\infty(M)$. In particular, the assignment:
  \begin{equation*}
    f \mapsto X_f
  \end{equation*}
  is an isomorphism of Lie algebras onto the Lie algebra of
  Hamiltonian fields.
\end{lemma}

Another important object in our discussion is given by the notion of a
foliation $\F$ of a manifold $M$. This is given as a decomposition
into connected submanifolds which is locally given by submersions.
More precisely, we define a foliated chart for $M$ as a smooth
submersion $\varphi : U \subseteq M \rightarrow \R^k$ from an open
subset of $M$ onto an open subset of $\R^k$. Given two such foliated
charts $\varphi, \psi$, defined on open subsets $U,V$ respectively, we
will say that they are compatible if there is a smooth diffeomorphism
$\xi : \varphi(U\cap V) \rightarrow \psi(U\cap V)$ such that
$\xi\circ\varphi = \psi$ on $U\cap V$. A foliated atlas for $M$ is a
family of compatible foliated charts whose domains cover $M$; note
that we also need $k$ above to be the same for all the foliated
charts. For any such foliated atlas, we define the plaques as the
connected components of the fibers of its foliated charts. With these
plaques we define the following equivalence relation in $M$:
\begin{eqnarray*}
          x\sim y &\iff& \mbox{there is a sequence of plaques }
                                (P_j)_{j=0}^l  \\
                        &&      \mbox{of the foliated atlas, such that }
                                x\in P_0,\ y\in P_l,   \\
                        &&      \mbox{and }
                                P_{j-1}\cap P_j\neq\phi \mbox{ for every }
                                j=1,\dots,l
\end{eqnarray*}
The equivalence classes of such an equivalence relation are called the
leaves of the foliation, which are easily seen to be submanifolds of
$M$. For further details on this definition and some of its
consequences and properties we refer to \cite{Candel-Conlon}. 

Finally, a foliation $\F$ in a symplectic manifold $M$ is called
Lagrangian if its leaves are Lagrangian submanifolds of $M$.

\section{Berezin's correspondence principle for bounded domains}
\label{sec:Ber}
In the rest of this section $D$ denotes a homogeneous bounded domain of
$\C^n$. We now recollect some facts on the analysis of $D$ leading to
Berezin's correspondence principle; we refer to \cite{Berezin} for
further details.

For every $h \in (0,1)$, let us denote by $\A_h^2(D)$ the weighted
Bergman space defined as the closed subspace of holomorphic functions
in $L^2(D,d\mu_h)$. Here, $d\mu_h$ denotes the weighted volume element
obtained from the Bergman kernel. If we denote by $B_D^{(h)} :
L^2(D,d\mu_h) \rightarrow \A_h^2(D)$ the orthogonal projection, then
for every $a \in L^\infty(D,d\mu_h)$  the Toeplitz operator
$T_a^{(h)}$ with (anti--Wick) symbol $a$ is given by the assignment:
\begin{eqnarray*}
  T_a^{(h)} : \A_h^2(D) &\rightarrow& \A_h^2(D) \\
  \varphi &\mapsto& B_D^{(h)}(a\varphi).
\end{eqnarray*}
For any such (anti--Wick) symbol $a$ and its associated Toeplitz
operator $T_a^{(h)}$, Berezin \cite{Berezin} constructed a (Wick)
symbol $\widetilde{a}_h : D \times \overline{D} \rightarrow \C$
defined so that the relation:
\begin{equation*}
  T_a^{(h)}(\varphi)(z) = 
  \int_D \widetilde{a}_h(z,\overline{\zeta})\varphi(\zeta)
  F_h(\zeta,\overline{\zeta}) d\mu(\zeta),
\end{equation*}
holds for every $\varphi \in \A_h^2(D)$, where $d\mu$ is the
(weightless) Bergman volume and $F_h$ is a suitable kernel defined in
terms of the Bergman kernel and depending on $h$.  This provides the
means to describe the algebra of Toeplitz operators as a suitable
algebra of functions. To achieve this, one defines a $*$-product of
two Wick symbols $\widetilde{a}_h, \widetilde{b}_h$ as the symbol
given by:
\begin{equation*}
  (\widetilde{a}_h * \widetilde{b}_h)(z,\overline{z}) =
  \int_D \widetilde{a}_h(z,\overline{\zeta})
  \widetilde{b}_h(\zeta,\overline{z})
  G_h(\zeta,\overline{\zeta},z,\overline{z}) d\mu(\zeta),
\end{equation*}
where $G_h$ is again some kernel defined in terms of the Bergman
kernel and depending on $h$. We denote by $\widetilde{\A}_h$ the
vector space of Wick symbols associated to anti--Wick symbols that
belong to $C^\infty_b(D)$. Then $\widetilde{\A}_h$ can be considered
as an algebra for the $*$-product defined above.

Berezin's correspondence principle is then stated as follows.

\begin{theorem}[Berezin \cite{Berezin}]\label{thm:Ber}
  Let $D$ be a homogeneous bounded domain of $\C^n$. Then, the map
  given by:
  \begin{eqnarray*}
    \T_h(C^\infty_b(D)) &\rightarrow& \widetilde{\A}_h \\
    T_a^{(h)} &\mapsto& \widetilde{a}_h
  \end{eqnarray*}
  is an isomorphism of algebras, where $\T_h(C^\infty_b(D))$ is the
  Toeplitz operator algebra defined by bounded smooth symbols.
  Furthermore, the following correspondence principle is satisfied:
  \begin{equation*}
    (\widetilde{a}_h * \widetilde{b}_h - \widetilde{b}_h *
    \widetilde{a}_h)(z,\overline{z}) = 
    ih\{a,b\}(z) + O(h^2),
  \end{equation*}
for every $h \in (0,1)$ and every $a,b \in C^\infty_b(D)$.
\end{theorem}

\section{Proofs of the main results}

For the sake of completeness, we present here the proof of
Proposition~\ref{prop:abelian}.

\begin{proof}[Proof of Proposition~\ref{prop:abelian}]
  For a given $a \in \A_\F(M)$, the condition of $a$ being constant
  along the leaves of $\F$ implies that:
  \begin{equation*}
   \omega(X_a,X) = da(X) = 0,
  \end{equation*}
  for every smooth vector field $X$ tangent to $\F$. Since the
  foliation $\F$ is Lagrangian, at every $p \in M$ the space $T_p\F$
  is a maximal isotropic subspace for $\omega$ and so the above
  identity shows that $(X_a)_p$ belongs to $T_p\F$. Hence, for every
  $a \in \A_\F(M)$ the vector field $X_a$ is tangent to the foliation
  $\F$. From this we conclude that:
  \begin{equation*}
    \{a,b\} = da(X_b) = 0,
  \end{equation*}
for every $a,b \in \A_\F(M)$.
\end{proof}

We now prove our result on commutative Poisson algebras and Lagrangian
foliations.

\begin{proof}[Proof of Theorem \ref{thm:lag-fol}]
  Let us consider a subspace $\A$ of $C^\infty(M)$ as in the
  hypotheses of Theorem \ref{thm:lag-fol}. For every $p \in M$ define
  the vector subspace of $T_pM$ given by:
  \begin{equation*}
    E_p = \{(X_a)_p : a \in \A\}.
  \end{equation*}
  We will now prove that $E = \cup_{p \in M} E_p$ is a smooth
  $n$-distribution over $M$; in other words, that in a neighborhood of
  every point the fibers of $E$ are spanned by $n$ smooth vector
  fields which are pointwise linearly indepent in such neighborhood.

  First note that since the assignment $a \mapsto X_a$ is linear,
  every set $E_p$ is a subspace of $T_pM$. Furthermore, being $\A$
  commutative for the Poisson brackets, it follows that:
  \begin{equation*}
   \omega(X_a,X_b) = \{a,b\} = 0,
  \end{equation*}
  for every $a,b \in \A$. In particular, $E_p$ is an isotropic
  subspace for $\omega$ and so has dimension at most $n$. 

  On the other hand, for every $p \in M$ we can choose smooth
  functions $a_1, \dots, a_n \in \A$ whose differentials are linearly
  independent at $p$. Hence, it follows from Lemma~\ref{lem:iso-omega}
  that the elements $(X_{a_1})_p, \dots, (X_{a_n})_p$ are also
  linearly independent at $p$, thus showing that $E_p$ has dimension
  exactly $n$. By continuity, the chosen vector fields $X_{a_1},
  \dots, X_{a_n}$ are linearly independent in a neighborhood of $p$
  and so their values span $E_q$ for every $q$ in such neighborhood.
  Hence, $E$ is indeed an $n$-distribution and our proof shows that
  its fibers are Lagrangian.

  By Lemma~\ref{lem:Lie}, the assignment $a \mapsto X_a$ is a
  homomorphism of Lie algebras, and so the commutativity of $\A$ with
  respect to the Poisson brackets implies that the vector fields $X_a$
  commute with each other for $a \in \A$. Since the latter span the
  distribution $E$ we conclude that $E$ is involutive and, by
  Frobenius' Theorem (see \cite{Warner}), it is integrable to some
  foliation $\F$. Note that $\F$ is necessarily Lagrangian.

  It is enough to prove that every $a \in \A$ is constant along the
  leaves of $\F$. But by hypothesis we have:
  \begin{equation*}
    da(X_b) = \{a,b\} = 0,
  \end{equation*}
  for every $a,b \in \A$. Since the vector fields $X_b$ ($b \in \A$)
  define the elements of $E$ at the fiber level we conclude that for
  any given $a \in \A$ we have:
  \begin{equation*}
    da(X) = 0
  \end{equation*}
  for every vector field $X$ tangent to $\F$. This implies that every
  $a \in \A$ is constant along the leaves of $\F$.
\end{proof}

Finally, we establish the necessity of having a Lagrangian foliation
underlying to every commutative algebra of Toeplitz operators with
sufficiently rich (anti--Wick) symbols.

\begin{proof}[Proof of the Main Theorem]
  For every $h \in (0,1)$, let us denote by $\widetilde{\A}_h(\A)$ the
  algebra of Wick symbols corresponding to anti--Wick symbols $a \in
  \A$. By the first part of Theorem~\ref{thm:Ber}, the algebra
  $\widetilde{\A}_h(\A)$ is commutative. Hence, by the correspondence
  principle stated in the second part of Theorem~\ref{thm:Ber} it
  follows that $\A$ is commutative with respect to the Poisson
  brackets $\{\cdot,\cdot\}$. The result now follows from
  Theorem~\ref{thm:lag-fol}.
\end{proof}

\end{document}